\documentclass[12pt]{amsart}
\usepackage{amsfonts}
\usepackage{amssymb,amsthm,amscd,amstext,amsmath,mathrsfs,latexsym}
\usepackage{tikz}
\usetikzlibrary{intersections,decorations.pathmorphing,decorations.fractals}
\newtheorem{theorem}{Theorem}
\newtheorem{lemma}{Lemma}

\newtheorem{proposition}{Proposition}

\newtheorem{definition}{Definition}

\begin{document}

\title[]
{Representations of the loop braid groups from braided tensor categories}

\author{Liang Chang}
\address{Chern Institute of Mathematics and LPMC\\
    Nankai University \\
    Tianjin, China 300071}
\email{changliang996@nankai.edu.cn}
\thanks{ }

\begin{abstract}
 The loop braid group is the motion group of unknotted oriented circles in $\mathbb{R}^3$. In this paper, we study their representations through the approach inspired by two dimensional topological phases of matter. In principle, the motion of loops in $\mathbb{R}^3$ reduces to the motions of points in a two dimensional sliced plane. We realize this physical picture in terms of braided tensor categories and their braid group representations. 
\end{abstract}

\maketitle

\section{Introduction}

In the topological phases of matter in two spacial dimensions, the point-like excitations, called \textit{anyons}, are modeled by a unitary modular category, equivalently, by a (2+1) topological quantum field theory (TQFT). The motion of anyons is described by the corresponding braid group representations from the modular category, which can yield topological quantum computation models \cite{FLW1, FKLW, Zhenghan}.

It is naturally to consider the excitations in 3d physical systems. By the spin-statistics theorem, point-like excitations in three spatial dimensions are all bosons or fermions. Switching two identical particles results in a plus/minus sign in their wave function. In this sense, the motion of point-like excitations only give the symmetric group representations. However, other than point-like excitations, there exist string-like excitations in the 3d topological phases of matter. The groups of their motion generalize the usual braid groups and were studied in \cite{dahm, Gold, lin} mathematically and physically in \cite{BWC, WL}, etc. In particular, the motion group of unknotted oriented circles in $\mathbb{R}^3$ are called\textit{ loop braid group}. While the unknotted circles are all linked to another base circle, their motion forms the \textit{necklace braid group}. Their representation theory have been studied through algebraic and gauge theory approaches (\cite{BCHPRS}, \cite{bard2}, \cite{BKMR}, etc.) .

In general, the topological phases of matter in three spacial dimensions are encoded by (3+1) TQFTs. Unlike the (2+1) TQFTs have been well understood by the rich theory of modular categories, the categorical theory for the (3+1) TQFTs have not been established completely. In order to obtain the representations of the motion group of links from the categorical perspective, we shall reduce three spacial dimensions to two spacial dimensions by focusing on the intersection of the links and an auxiliary plane. From this view, the motion of links is simulated by the motion of the intersecting points. Although these intersections can not recover all motions for any links, in some case they can still completely present the generators of some motion groups. In this paper, we will show the loop braid groups can be realized by the motion of pairs of points. As a result, we obtain loop braid group representations from the braid group representation based on braid tensor categories. 

This paper is organized as follows. In Section 2, we first review the notion of loop braid groups. Then we will show that the double-strand braids have two type of braidings and almost realize the loop braids in Section 3. Finally, we obtain loop braid group representation whenever two anyons fuse into several bosons or fermions.

\section{Loop braid groups}

The $n$-component loop braid is geometrically illustrated as the motions of $n$ unknotted oriented circles in $\mathbb{R}^3$.  We can present the generating loop braids diagrammatically via the trajectories of circles in the time direction as follows. 

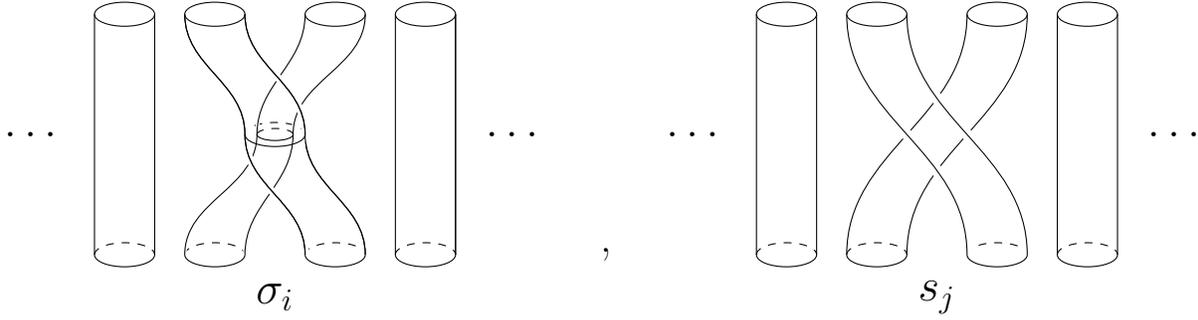
\[
\begin{tikzpicture}[scale=0.8]
\begin{scope}[xshift=-5.5cm]
\begin{scope}[xshift=-1.5cm, yshift=-2cm]
\draw (0,0) arc (180:360:{0.5} and {0.2});
\draw[dashed] (0,0) arc (180:0:{0.5} and {0.2});
\end{scope}
\begin{scope}[xshift=0.5cm, yshift=-2cm]
\draw (0,0) arc (180:360:{0.5} and {0.2});
\draw[dashed] (0,0) arc (180:0:{0.5} and {0.2});
\end{scope}
\draw [dashed] (0.5,0) arc (0:180:{0.5} and {0.2});
\draw [line width=0.5mm, white] (-1.5,-2) to [out=90, in=-90] (-0.3,0) to [out=90, in=-90] (0.5,2); 
\draw [line width=0.5mm, white] (-0.5,-2) to [out=90, in=-90] (0.3,0) to [out=90, in=-90] (1.5,2); 
\draw (-1.5,-2) to [out=90, in=-90] (-0.3,0) to [out=90, in=-90] (0.5,2); 
\draw (-0.5,-2) to [out=90, in=-90] (0.3,0) to [out=90, in=-90] (1.5,2); 
\draw (-0.3,0) arc (180:360:{0.3} and {0.1});
\draw[dashed] (0.3,0) arc (0:180:{0.3} and {0.1});
\draw [line width=1mm, white] (1.5,-2) to [out=90, in=-90] (0.5,0) to [out=90, in=-90] (-0.5,2); 
\draw [line width=1mm, white] (0.5,-2) to [out=90, in=-90] (-0.5,0) to [out=90, in=-90] (-1.5,2);
\draw (1.5,-2) to [out=90, in=-90] (0.5,0) to [out=90, in=-90] (-0.5,2); 
\draw (0.5,-2) to [out=90, in=-90] (-0.5,0) to [out=90, in=-90] (-1.5,2); 
\draw[line width=1mm, white] (0.5,0) arc (180:360:{0.5} and {0.2});
\draw (-0.5,0) arc (180:360:{0.5} and {0.2});
\draw (1.5,-2) to [out=90, in=-90] (0.5,0) to [out=90, in=-90] (-0.5,2); 
\draw (0.5,-2) to [out=90, in=-90] (-0.5,0) to [out=90, in=-90] (-1.5,2); 
\begin{scope}[xshift=-1.5cm, yshift=2cm]
\draw (0,0) arc (180:360:{0.5} and {0.2});
\draw (0,0) arc (180:0:{0.5} and {0.2});
\end{scope}
\begin{scope}[xshift=0.5cm, yshift=2cm]
\draw (0,0) arc (180:360:{0.5} and {0.2});
\draw (0,0) arc (180:0:{0.5} and {0.2});
\end{scope}
\draw (3,-2) to (3,2); 
\begin{scope}[xshift=-3cm]
\begin{scope}[yshift=-2cm]
\draw (0,0) arc (180:360:{0.5} and {0.2});
\draw[dashed] (0,0) arc (180:0:{0.5} and {0.2});
\end{scope}
\begin{scope}[yshift=2cm]
\draw (0,0) arc (180:360:{0.5} and {0.2});
\draw (0,0) arc (180:0:{0.5} and {0.2});
\end{scope}
\draw (0,-2) to (0,2); 
\draw (1,-2) to (1,2); 
\end{scope}
\begin{scope}[xshift=2cm]
\begin{scope}[yshift=-2cm]
\draw (0,0) arc (180:360:{0.5} and {0.2});
\draw[dashed] (0,0) arc (180:0:{0.5} and {0.2});
\end{scope}
\begin{scope}[yshift=2cm]
\draw (0,0) arc (180:360:{0.5} and {0.2});
\draw (0,0) arc (180:0:{0.5} and {0.2});
\end{scope}
\draw (0,-2) to (0,2); 
\draw (1,-2) to (1,2); 
\end{scope}
\node (k) at (-4,0) {\Large $\dots$};
\node (k) at (4,0) {\Large $\dots$};
\node (k) at (0,-2.7) {\Large $\sigma_i$};
\end{scope}
\node (k) at (0,-2) {\Large ,};
\begin{scope}[xshift=5.5cm]
\begin{scope}[xshift=-1.5cm, yshift=-2cm]
\draw (0,0) arc (180:360:{0.5} and {0.2});
\draw[dashed] (0,0) arc (180:0:{0.5} and {0.2});
\end{scope}
\begin{scope}[xshift=0.5cm, yshift=-2cm]
\draw (0,0) arc (180:360:{0.5} and {0.2});
\draw[dashed] (0,0) arc (180:0:{0.5} and {0.2});
\end{scope}
\draw (-1.5,-2) to [out=90, in=-90] (0.5,2); 
\draw (-0.5,-2) to [out=90, in=-90] (1.5,2); 
\draw [line width=1.5mm, white] (1.5,-2) to [out=90, in=-90] (-0.5,2); 
\draw [line width=1.5mm, white] (0.5,-2) to [out=90, in=-90] (-1.5,2);
\draw (1.5,-2) to [out=90, in=-90] (-0.5,2); 
\draw (0.5,-2) to [out=90, in=-90] (-1.5,2); 
\begin{scope}[xshift=-1.5cm, yshift=2cm]
\draw (0,0) arc (180:360:{0.5} and {0.2});
\draw (0,0) arc (180:0:{0.5} and {0.2});
\end{scope}
\begin{scope}[xshift=0.5cm, yshift=2cm]
\draw (0,0) arc (180:360:{0.5} and {0.2});
\draw (0,0) arc (180:0:{0.5} and {0.2});
\end{scope}
\begin{scope}[xshift=-3cm]
\begin{scope}[yshift=-2cm]
\draw (0,0) arc (180:360:{0.5} and {0.2});
\draw[dashed] (0,0) arc (180:0:{0.5} and {0.2});
\end{scope}
\begin{scope}[yshift=2cm]
\draw (0,0) arc (180:360:{0.5} and {0.2});
\draw (0,0) arc (180:0:{0.5} and {0.2});
\end{scope}
\draw (0,-2) to (0,2); 
\draw (1,-2) to (1,2); 
\end{scope}
\begin{scope}[xshift=2cm]
\begin{scope}[yshift=-2cm]
\draw (0,0) arc (180:360:{0.5} and {0.2});
\draw[dashed] (0,0) arc (180:0:{0.5} and {0.2});
\end{scope}
\begin{scope}[yshift=2cm]
\draw (0,0) arc (180:360:{0.5} and {0.2});
\draw (0,0) arc (180:0:{0.5} and {0.2});
\end{scope}
\draw (0,-2) to (0,2); 
\draw (1,-2) to (1,2); 
\end{scope}
\node (k) at (-4,0) {\Large $\dots$};
\node (k) at (4,0) {\Large $\dots$};
\node (k) at (0,-2.7) {\Large $s_j$};
\end{scope}
\end{tikzpicture}
\]

These diagrams shows two types of generators of loop braid groups. The generator $\sigma_i$ is to pass the $i$-th circle under and through the  $(i+1)$-th circle ending with the two circles' positions interchanged. The generator $s_i$ is simply to switch the $i$-th and $(i+1)$-th circles. Algebraically, the loop braid group is defined by its generators and relations. 
\begin{definition}
	The $n$-component \textbf{loop braid group} $\text{LB}_n$ is the group generated by $\sigma_i, s_j$ for $1\leq i, j \leq n$ satisfying the following relations:
	\begin{enumerate}
		\item[(B1)] $\sigma_i\sigma_{i+1}\sigma_i=\sigma_{i+1}\sigma_i\sigma_{i+1} $
		\item[(B2)] $\sigma_i\sigma_{j}=\sigma_j\sigma_{i} ~~\text{for}~~ |i-j|>1$
		\item[(S1)] $s_js_{j+1}s_j=s_{j+1}s_js_{j+1} $
		\item[(S2)] $s_j^2=1 $
		\item[(S3)] $s_i s_{j}=s_j s_{i} ~~\text{for}~~ |i-j|>1$
		\item[(M1)] $s_is_{i+1}\sigma_i=\sigma_{i+1}s_is_{i+1}$
		\item[(M2)] $\sigma_i\sigma_{i+1}s_i=s_{i+1}\sigma_i\sigma_{i+1}$
		\item[(M3)] $\sigma_i s_j=s_j \sigma_i ~~\text{for}~~ |i-j|>1$
	\end{enumerate}
\end{definition}

This set of generators and relations was introduced in \cite{lin} and \cite{BWC}. It is as the same as the definition of welded braid groups in the virtual knot theory \cite{Kauf}. There double points are allowed in the virtual knot diagrams. That is, the welded braid groups is generated by the following set of crossings modulo the same relations as in Definition 1. Therefore it is possible to study the loop braid group $\text{LB}_n$ as extension of the braid group  $\text{B}_n$. We refer readers to \cite{journey} for the equivalence of various definitions of $\text{LB}_n$ and welded diagram approaches.

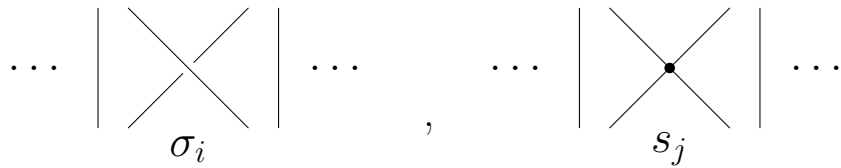
\[
\begin{tikzpicture}[scale=0.4]
\begin{scope}[xshift=-8cm]
\draw (-2,-2) to (2,2); 
\draw [line width=2mm, white] (2,-2) to (-2,2); 
\draw (2,-2) to (-2,2); 
\draw (-3,-2) to (-3,2); 
\draw (3,-2) to (3,2); 
\node (k) at (-5,0) {\Large $\dots$};
\node (k) at (5,0) {\Large $\dots$};
\node (k) at (0,-2.7) {\Large $\sigma_i$};
\end{scope}
\node (k) at (0,-2) {\Large ,};
\begin{scope}[xshift=8cm]
\draw (-2,-2) to (2,2); 
\draw (2,-2) to (-2,2); 
\draw (-3,-2) to (-3,2); 
\draw (3,-2) to (3,2); 
\node (k) at (-5,0) {\Large $\dots$};
\node (k) at (5,0) {\Large $\dots$};
\fill[black, opacity=1] (0,0) circle (5pt);
\node (k) at (0,-2.7) {\Large $s_j$};
\end{scope}
\end{tikzpicture}
\]

\section{two types of double-strand braidings}

When a circle intersects a plane transversely, the intersection is a pair of points. While the circle travels parallel to the plane, its motion can be captured by the motion of these two points. In particular, the generating motions $\sigma_i$ and $s_j$ are shown below.

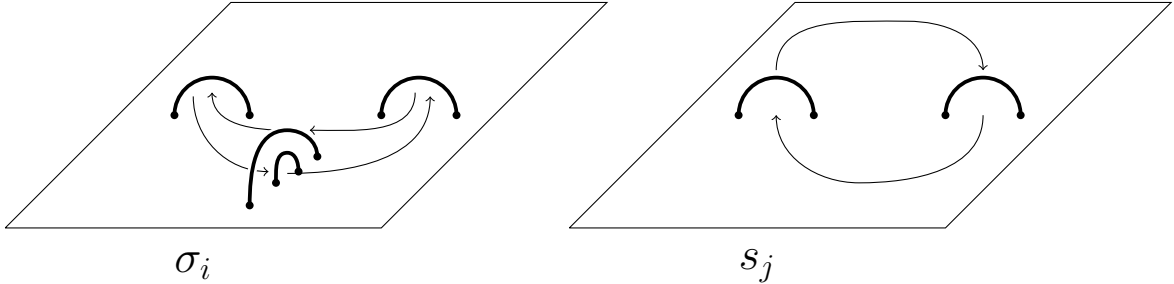
\[
\begin{tikzpicture}[scale=0.5]
\begin{scope}
\draw (-2,-1) to [out=45, in=-135] (4,5) to [out=0, in=180] (14,5)to [out=-135, in=45] (8,-1) to [out=180, in=0] (-2,-1);
\draw [line width=0.5mm] (2.5,2) to [out=90, in=180] (3.5,3) to [out=0, in=90] (4.5,2);
\draw [line width=0.5mm] (8,2) to [out=90, in=180] (9,3) to [out=0, in=90] (10,2);
\draw [->] (3,2.5) to [out=-90, in=180] (5,0.5);
\draw [->]  (5.5,0.45) to [out=0, in=-90] (9.3,2.5);
\draw [->] (8.9,2.6) to [out=-90, in=0] (6.1,1.6);
\draw [->]  (5.5,1.6) to [out=180, in=-90] (3.5,2.6);
\draw [white, line width=1mm] (5.2,0.2) to [out=90, in=180] (5.5,1) to [out=0, in=90] (5.8,0.5);
\draw [line width=0.5mm] (5.2,0.2) to [out=90, in=180] (5.5,1) to [out=0, in=90] (5.8,0.5);
\draw [white, line width=1.2mm] (4.5,-0.4) to [out=90, in=180] (5.5,1.6) to [out=0, in=90] (6.3,0.9);
\draw [line width=0.5mm] (4.5,-0.4) to [out=90, in=180] (5.5,1.6) to [out=0, in=90] (6.3,0.9);
\fill[black, opacity=1] (2.5,2) circle (3pt);
\fill[black, opacity=1] (4.5,2) circle (3pt);
\fill[black, opacity=1] (8,2) circle (3pt);
\fill[black, opacity=1] (10,2) circle (3pt);
\fill[black, opacity=1] (5.2,0.2) circle (3pt);
\fill[black, opacity=1] (5.8,0.5) circle (3pt);
\fill[black, opacity=1] (4.5,-0.4) circle (3pt);
\fill[black, opacity=1] (6.3,0.9) circle (3pt);
\node (k) at (3,-2) {\Large $\sigma_i$};
\end{scope}
\begin{scope}[xshift=16cm]
\draw (-3,-1) to [out=45, in=-135] (3,5) to [out=0, in=180] (13,5)to [out=-135, in=45] (7,-1) to [out=180, in=0] (-3,-1);
\draw [line width=0.5mm] (1.5,2) to [out=90, in=180] (2.5,3) to [out=0, in=90] (3.5,2);
\draw [line width=0.5mm] (7,2) to [out=90, in=180] (8,3) to [out=0, in=90] (9,2);
\draw [->] (8,2) to [out=-90, in=0] (4.7,0.2) to [out=180, in=-90] (2.5,2);
\draw [<-] (8,3.2) to [out=90, in=0] (6,4.5) to [out=180, in=90] (2.5,3.2);
\fill[black, opacity=1] (1.5,2) circle (3pt);
\fill[black, opacity=1] (3.5,2) circle (3pt);
\fill[black, opacity=1] (7,2) circle (3pt);
\fill[black, opacity=1] (9,2) circle (3pt);
\node (k) at (2,-2) {\Large $s_j$};
\end{scope}
\end{tikzpicture}
\]

As a result, we reduce the loop braid groups to the braid groups. That is, the trajectories of the intersecting points in the  time direction form the braids in $\text{B}_{2n}$. Correspondingly, the motions $\sigma_i$ and $s_j$ give the following braids $\widetilde{\sigma}_i$ and $\widetilde{s}_j$  in $\text{B}_{2n}$, repsectively. 

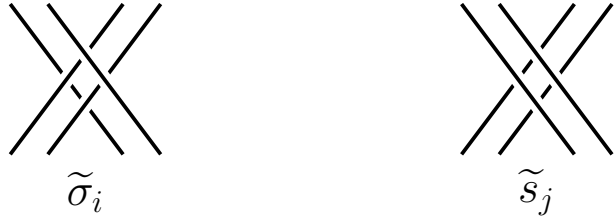
\[
\begin{tikzpicture}[scale=0.5]
\begin{scope}[xshift=-2cm]
\draw[line width=0.5mm] (1,0)--(-2,4);
\draw[white, line width=2mm] (-2,0)--(1,4); 
\draw[line width=0.5mm] (-2,0)--(1,4); 
\draw[white, line width=2mm] (-1,0)--(2,4); 
\draw[line width=0.5mm] (-1,0)--(2,4);
\draw[white, line width=2mm] (2,0)--(-1,4); 
\draw[line width=0.5mm] (2,0)--(-1,4); 
\node (x) at (0,-1) {\Large $\widetilde{\sigma}_i$};
\end{scope}
\begin{scope}[xshift=10cm]
\draw[line width=0.5mm] (-2,0)--(1,4); 
\draw[line width=0.5mm] (-1,0)--(2,4);
\draw[white, line width=2mm] (2,0)--(-1,4); 
\draw[white, line width=2mm] (1,0)--(-2,4);
\draw[line width=0.5mm] (2,0)--(-1,4); 
\draw[line width=0.5mm] (1,0)--(-2,4);
\node (x) at (0,-1) {\Large $\widetilde{s}_j$};
\end{scope}	 
\end{tikzpicture}
\]

\begin{proposition}
	The braidings $\widetilde{\sigma}_i$ and $\widetilde{s}_j$ satisfy the braid relations: $\widetilde{\sigma}_{i}\widetilde{\sigma}_{i+1}\widetilde{\sigma}_{i}=\widetilde{\sigma}_{i+1}\widetilde{\sigma}_{i}\widetilde{\sigma}_{i+1}$, $\widetilde{s}_{j}\widetilde{s}_{j+1}\widetilde{s}_{j}=\widetilde{s}_{j+1}\widetilde{s}_{j}\widetilde{s}_{j+1}$,  the mixed relation (M1): $\widetilde{s}_{k}\widetilde{s}_{k+1}\widetilde{\sigma}_{k}=\widetilde{\sigma}_{k+1}\widetilde{s}_{k}\widetilde{s}_{k+1}$, and all the far commutativities (B2), (S3), (M3). 
\end{proposition}
\begin{proof}
	It is straightforwards to check the following braid diagram equalities through several Reidemeister III moves.
	
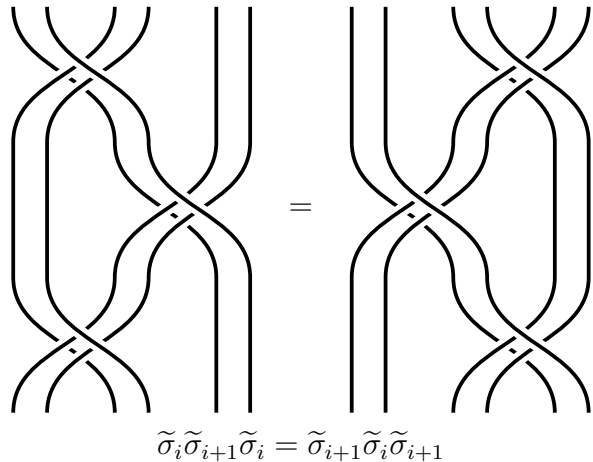
\[
	\begin{tikzpicture}[scale=0.45]
	\begin{scope}
	\draw [line width=0.5mm] (7,0) to [out=90, in=-90] (7,4) to [out=90, in=-90] (4,8)to [out=90, in=-90] (1,12); 
	\draw [white, line width=1.8mm] (4,0) to [out=90, in=-90] (1,4) to [out=90, in=-90] (1,8) to [out=90, in=-90] (4,12); 
	\draw [line width=0.5mm] (4,0) to [out=90, in=-90] (1,4) to [out=90, in=-90] (1,8) to [out=90, in=-90] (4,12); 
	\draw[white, line width=1.8mm] (1,0) to [out=90, in=-90] (4,4) to [out=90, in=-90] (7,8) to [out=90, in=-90] (7,12);
	\draw[line width=0.5mm] (1,0) to [out=90, in=-90] (4,4) to [out=90, in=-90] (7,8) to [out=90, in=-90] (7,12);
	\draw[white, line width=1.8mm] (2,0) to [out=90, in=-90] (5,4) to [out=90, in=-90] (8,8);
	\draw[line width=0.5mm] (2,0) to [out=90, in=-90] (5,4) to [out=90, in=-90] (8,8) to [out=90, in=-90] (8,12);
	\draw[white, line width=1.8mm] (5,0) to [out=90, in=-90] (2,4) to [out=90, in=-90] (2,8) to [out=90, in=-90] (5,12);
	\draw[line width=0.5mm] (5,0) to [out=90, in=-90] (2,4) to [out=90, in=-90] (2,8) to [out=90, in=-90] (5,12);
	\draw[white, line width=1.8mm] (8,0) to [out=90, in=-90] (8,4) to [out=90, in=-90] (5,8) to [out=90, in=-90] (2,12);
	\draw[line width=0.5mm] (8,0) to [out=90, in=-90] (8,4) to [out=90, in=-90] (5,8) to [out=90, in=-90] (2,12);
	\end{scope}
	\begin{scope}[xshift=9.5cm]
	\node (k) at (0,6) {$=$};
	\end{scope}
	\begin{scope}[xshift=10cm]
	\draw [line width=0.5mm] (7,0) to [out=90, in=-90] (4,4) to [out=90, in=-90] (1,8)to [out=90, in=-90] (1,12); 
	\draw [white, line width=1.8mm] (4,0) to [out=90, in=-90] (7,4) to [out=90, in=-90] (7,8) to [out=90, in=-90] (4,12); 
	\draw [line width=0.5mm] (4,0) to [out=90, in=-90] (7,4) to [out=90, in=-90] (7,8) to [out=90, in=-90] (4,12);
	\draw[white, line width=1.8mm] (1,0) to [out=90, in=-90] (1,4) to [out=90, in=-90] (4,8) to [out=90, in=-90] (7,12);
	\draw[line width=0.5mm] (1,0) to [out=90, in=-90] (1,4) to [out=90, in=-90] (4,8) to [out=90, in=-90] (7,12);
	\draw[white, line width=1.8mm] (2,0) to [out=90, in=-90] (2,4) to [out=90, in=-90] (5,8) to [out=90, in=-90] (8,12);
	\draw[line width=0.5mm] (2,0) to [out=90, in=-90] (2,4) to [out=90, in=-90] (5,8) to [out=90, in=-90] (8,12);
	\draw[white, line width=1.8mm] (5,0) to [out=90, in=-90] (8,4) to [out=90, in=-90] (8,8) to [out=90, in=-90] (5,12);
	\draw[line width=0.5mm] (5,0) to [out=90, in=-90] (8,4) to [out=90, in=-90] (8,8) to [out=90, in=-90] (5,12);
	\draw[white, line width=1.8mm] (8,0) to [out=90, in=-90] (5,4) to [out=90, in=-90] (2,8) to [out=90, in=-90] (2,12);
	\draw[line width=0.5mm] (8,0) to [out=90, in=-90] (5,4) to [out=90, in=-90] (2,8) to [out=90, in=-90] (2,12);
	\end{scope}
	\begin{scope}[xshift=9.5cm, yshift=-7cm]
	\node (k) at (0,6) {$\widetilde{\sigma}_{i}\widetilde{\sigma}_{i+1}\widetilde{\sigma}_{i}=\widetilde{\sigma}_{i+1}\widetilde{\sigma}_{i}\widetilde{\sigma}_{i+1}$};
	\end{scope}
	\end{tikzpicture}
	\]
	\[
	\begin{tikzpicture}[scale=0.45]
	\begin{scope}
	\draw [line width=0.5mm] (7,0) to [out=90, in=-90] (7,4) to [out=90, in=-90] (4,8)to [out=90, in=-90] (1,12); 
	\draw[line width=0.5mm] (8,0) to [out=90, in=-90] (8,4) to [out=90, in=-90] (5,8) to [out=90, in=-90] (2,12);
	\draw [white, line width=1.8mm] (4,0) to [out=90, in=-90] (1,4) to [out=90, in=-90] (1,8) to [out=90, in=-90] (4,12); 
	\draw [line width=0.5mm] (4,0) to [out=90, in=-90] (1,4) to [out=90, in=-90] (1,8) to [out=90, in=-90] (4,12); 
	\draw[white, line width=1.8mm] (5,0) to [out=90, in=-90] (2,4) to [out=90, in=-90] (2,8) to [out=90, in=-90] (5,12);
	\draw[line width=0.5mm] (5,0) to [out=90, in=-90] (2,4) to [out=90, in=-90] (2,8) to [out=90, in=-90] (5,12);
	\draw[white, line width=1.8mm] (1,0) to [out=90, in=-90] (4,4) to [out=90, in=-90] (7,8) to [out=90, in=-90] (7,12);
	\draw[line width=0.5mm] (1,0) to [out=90, in=-90] (4,4) to [out=90, in=-90] (7,8) to [out=90, in=-90] (7,12);
	\draw[white, line width=1.8mm] (2,0) to [out=90, in=-90] (5,4) to [out=90, in=-90] (8,8);
	\draw[line width=0.5mm] (2,0) to [out=90, in=-90] (5,4) to [out=90, in=-90] (8,8) to [out=90, in=-90] (8,12);
	\end{scope}
	\begin{scope}[xshift=9.5cm]
	\node (k) at (0,6) {$=$};
	\end{scope}
	\begin{scope}[xshift=10cm]
	\draw [line width=0.5mm] (7,0) to [out=90, in=-90] (4,4) to [out=90, in=-90] (1,8)to [out=90, in=-90] (1,12); 
	\draw[line width=0.5mm] (8,0) to [out=90, in=-90] (5,4) to [out=90, in=-90] (2,8) to [out=90, in=-90] (2,12);
	\draw [white, line width=1.8mm] (4,0) to [out=90, in=-90] (7,4) to [out=90, in=-90] (7,8) to [out=90, in=-90] (4,12); 
	\draw [line width=0.5mm] (4,0) to [out=90, in=-90] (7,4) to [out=90, in=-90] (7,8) to [out=90, in=-90] (4,12);
	\draw[white, line width=1.8mm] (5,0) to [out=90, in=-90] (8,4) to [out=90, in=-90] (8,8) to [out=90, in=-90] (5,12);
	\draw[line width=0.5mm] (5,0) to [out=90, in=-90] (8,4) to [out=90, in=-90] (8,8) to [out=90, in=-90] (5,12);
	\draw[white, line width=1.8mm] (1,0) to [out=90, in=-90] (1,4) to [out=90, in=-90] (4,8) to [out=90, in=-90] (7,12);
	\draw[line width=0.5mm] (1,0) to [out=90, in=-90] (1,4) to [out=90, in=-90] (4,8) to [out=90, in=-90] (7,12);
	\draw[white, line width=1.8mm] (2,0) to [out=90, in=-90] (2,4) to [out=90, in=-90] (5,8) to [out=90, in=-90] (8,12);
	\draw[line width=0.5mm] (2,0) to [out=90, in=-90] (2,4) to [out=90, in=-90] (5,8) to [out=90, in=-90] (8,12);
	\end{scope}
	\begin{scope}[xshift=9.5cm, yshift=-7cm]
	\node (k) at (0,6) {$\widetilde{s}_{j}\widetilde{s}_{j+1}\widetilde{s}_{j}=\widetilde{s}_{j+1}\widetilde{s}_{j}\widetilde{s}_{j+1}$};
	\end{scope}
	\end{tikzpicture}
	\]
	\[
	\begin{tikzpicture}[scale=0.5]
	\begin{scope}
	\draw [line width=0.5mm] (7,0) to [out=90, in=-90] (7,4) to [out=90, in=-90] (4,8)to [out=90, in=-90] (1,12); 
	\draw [white, line width=1.8mm] (4,0) to [out=90, in=-90] (1,4) to [out=90, in=-90] (1,8) to [out=90, in=-90] (4,12); 
	\draw [line width=0.5mm] (4,0) to [out=90, in=-90] (1,4) to [out=90, in=-90] (1,8) to [out=90, in=-90] (4,12); 
	\draw[white, line width=1.8mm] (5,0) to [out=90, in=-90] (2,4) to [out=90, in=-90] (2,8) to [out=90, in=-90] (5,12);
	\draw[line width=0.5mm] (5,0) to [out=90, in=-90] (2,4) to [out=90, in=-90] (2,8) to [out=90, in=-90] (5,12);
	\draw[white, line width=1.8mm] (8,0) to [out=90, in=-90] (8,4) to [out=90, in=-90] (5,8) to [out=90, in=-90] (2,12);
	\draw[line width=0.5mm] (8,0) to [out=90, in=-90] (8,4) to [out=90, in=-90] (5,8) to [out=90, in=-90] (2,12);
	\draw[white, line width=1.8mm] (1,0) to [out=90, in=-90] (4,4) to [out=90, in=-90] (7,8) to [out=90, in=-90] (7,12);
	\draw[line width=0.5mm] (1,0) to [out=90, in=-90] (4,4) to [out=90, in=-90] (7,8) to [out=90, in=-90] (7,12);
	\draw[white, line width=1.8mm] (2,0) to [out=90, in=-90] (5,4) to [out=90, in=-90] (8,8);
	\draw[line width=0.5mm] (2,0) to [out=90, in=-90] (5,4) to [out=90, in=-90] (8,8) to [out=90, in=-90] (8,12);
	\end{scope}
	\begin{scope}[xshift=9.5cm]
	\node (k) at (0,6) {$=$};
	\end{scope}
	\begin{scope}[xshift=10cm]
	\draw [line width=0.5mm] (7,0) to [out=90, in=-90] (4,4) to [out=90, in=-90] (1,8)to [out=90, in=-90] (1,12); 
	\draw [white, line width=1.8mm] (4,0) to [out=90, in=-90] (7,4) to [out=90, in=-90] (7,8) to [out=90, in=-90] (4,12); 
	\draw [line width=0.5mm] (4,0) to [out=90, in=-90] (7,4) to [out=90, in=-90] (7,8) to [out=90, in=-90] (4,12);
	\draw[white, line width=1.8mm] (5,0) to [out=90, in=-90] (8,4) to [out=90, in=-90] (8,8) to [out=90, in=-90] (5,12);
	\draw[line width=0.5mm] (5,0) to [out=90, in=-90] (8,4) to [out=90, in=-90] (8,8) to [out=90, in=-90] (5,12);
	\draw[white, line width=1.8mm] (8,0) to [out=90, in=-90] (5,4) to [out=90, in=-90] (2,8) to [out=90, in=-90] (2,12);
	\draw[line width=0.5mm] (8,0) to [out=90, in=-90] (5,4) to [out=90, in=-90] (2,8) to [out=90, in=-90] (2,12);
	\draw[white, line width=1.8mm] (1,0) to [out=90, in=-90] (1,4) to [out=90, in=-90] (4,8) to [out=90, in=-90] (7,12);
	\draw[line width=0.5mm] (1,0) to [out=90, in=-90] (1,4) to [out=90, in=-90] (4,8) to [out=90, in=-90] (7,12);
	\draw[white, line width=1.8mm] (2,0) to [out=90, in=-90] (2,4) to [out=90, in=-90] (5,8) to [out=90, in=-90] (8,12);
	\draw[line width=0.5mm] (2,0) to [out=90, in=-90] (2,4) to [out=90, in=-90] (5,8) to [out=90, in=-90] (8,12);
	\end{scope}
	\begin{scope}[xshift=9.5cm, yshift=-7cm]
	\node (k) at (0,6) {$\widetilde{s}_{k}\widetilde{s}_{k+1}\widetilde{\sigma}_{k}=\widetilde{\sigma}_{k+1}\widetilde{s}_{k}\widetilde{s}_{k+1}$};
	\end{scope}
	\end{tikzpicture}
	\]
	
	 Finally, (B2), (S3), (M3) are obvious from the far commutativities of $\text{B}_{2n}$.
\end{proof}

Therefore, the motions of pairs of points give rise to a subgroup $\widetilde{\text{LB}}_n$ of $\text{B}_{2n}$ which consists of the braids of double-strands and is generated by $\widetilde{\sigma}_{i}$ and $\widetilde{s}_j$. The following proposition shows $\widetilde{\text{LB}}_n$ is the image of $\text{B}_n$ embeded in $\text{B}_{2n}$ in two ways.  Similar construction appeared in \cite{BH}. Here we aim to get linear representations of $\text{LB}_n$ in terms of braid category data.

However, the symmetric relation (S2) $\widetilde{s}_j^2=1$ and the mixed relation (M2) $\widetilde{\sigma}_{i}\widetilde{\sigma}_{i+1}\widetilde{s}_{i}=\widetilde{s}_{i+1}\widetilde{\sigma}_{i}\widetilde{\sigma}_{i+1}$ are not automatically satisfied. From physical point of view, the symmetric relation means that the pair of points fuse into some bosons or fermions. But this is not true in the 2d topological order where more general anyons other than bosons or fermions may appear after the fusion process. Therefore, the symmetric relation should naturally be a sufficient condition for the existence of $\text{LB}_n$ representations. In fact, we show that (M2) results from the symmetric relation in the next section.

\section{The loop braid group representations}

\textit{Braided tensor categories} (\cite{EGNO}) are rich resource of the braid group representations. These are tensor categories equipped with
a family of natural transformations, called \textit{braidings}. That is, for each pair of objects $x$, $y$ of $\mathcal{C}$,  one has a natural transformation $c_{x,y}\in \text{Hom}(x, y)$ satisfying certain compatibility
 equations that are regarded as the categorical version of Yang-Baxter equation. For an object $x$, 
 $R_i=\text{id}^{\otimes(i-1)}\otimes c_{x,x}\otimes \text{id}^{\otimes(n-i-1)}$
 induces a homomorphism $B_n\rightarrow \text{End}(x^{\otimes n})$ by $\sigma_i\mapsto R_i$.  Applying diagram calculus for braided categories, we have $\widetilde{\sigma}_i$ and $\widetilde{s}_j$ as morphisms in $\mathcal{C}$ expressed in terms of  the braiding $c$. 
 
 Let $x$, $y$ be objects such that the double braiding of $x\otimes y$ is trivial, i.e., $$c_{x\otimes y, x\otimes y}\circ c_{x\otimes y, x\otimes y}=\text{id}_{x\otimes y, x\otimes y}.$$
 This is the symmetric equation (S2) in the categorical setting. Equivalently, it means that the $(x, y)$-labeled double-strands can pass through the crossing in the braid diagrams.
 
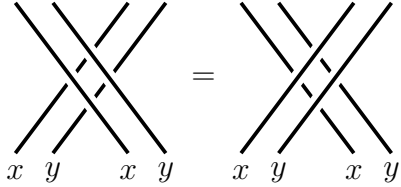
\[
 \begin{tikzpicture}[scale=0.5]
 \begin{scope}
 \draw[line width=0.5mm] (-2,0)--(1,4); 
 \draw[line width=0.5mm] (-1,0)--(2,4);
 \draw[white, line width=2mm] (2,0)--(-1,4); 
 \draw[white, line width=2mm] (1,0)--(-2,4);
 \draw[line width=0.5mm] (2,0)--(-1,4); 
 \draw[line width=0.5mm] (1,0)--(-2,4);
 \node (x) at (-2,-0.5) {$x$};
 \node (x) at (-1,-0.5) {$y$};
 \node (x) at (1,-0.5) {$x$};
 \node (x) at (2,-0.5) {$y$};
 \end{scope}
 \begin{scope}[xshift=3cm]
 \node (k) at (0,2) {$=$};
 \end{scope}
 \begin{scope}[xshift=6cm]
 \draw[line width=0.5mm] (2,0)--(-1,4); 
 \draw[line width=0.5mm] (1,0)--(-2,4);
 \draw[white, line width=2mm] (-2,0)--(1,4); 
 \draw[white, line width=2mm] (-1,0)--(2,4);
 \draw[line width=0.5mm] (-2,0)--(1,4); 
 \draw[line width=0.5mm] (-1,0)--(2,4);
 \node (x) at (-2,-0.5) {$x$};
 \node (x) at (-1,-0.5) {$y$};
 \node (x) at (1,-0.5) {$x$};
 \node (x) at (2,-0.5) {$y$};
 \end{scope}	 
 \end{tikzpicture}
 \]
 
The following theorem tells that this trivial double braiding suffices to imply the mixed relation $(\text{M2})$ and induce a loop braid group representation. 
 
 \begin{theorem}
 	Let $x$ and $y$ be objects in a braided tensor category $\mathcal{C}$ such that the double braiding of $x\otimes y$ is trivial. Then the braidings $\widetilde{\sigma}_i$ and $\widetilde{s}_j$ induces a homomorphism $\text{LB}_n\rightarrow\text{End}((x\otimes y)^{\otimes n})$.
 \end{theorem}
 
 This is the corollary of Proposition 1 and Lemma 1 below which implies the mixed relation (M2). Recall that in \cite{BCHPRS} the notion of \textit{standard extension} is introduced for the representation of $\text{LB}_3$ with $\sigma_1\sigma_2=ks_1s_2$ for some constant $k$. Now Lemma 1 asserts that the $\text{LB}_3$ representations in Theorem 1 are all standard. 
 
 \begin{lemma}
 	If $x\otimes y$ has trivial double braiding, then $\widetilde{\sigma}_i\widetilde{\sigma}_{i+1}=\widetilde{s}_i\widetilde{s}_{i+1}$
 \end{lemma}
 \begin{proof}
 	It can be verified by the following diagram calculus. 
 	
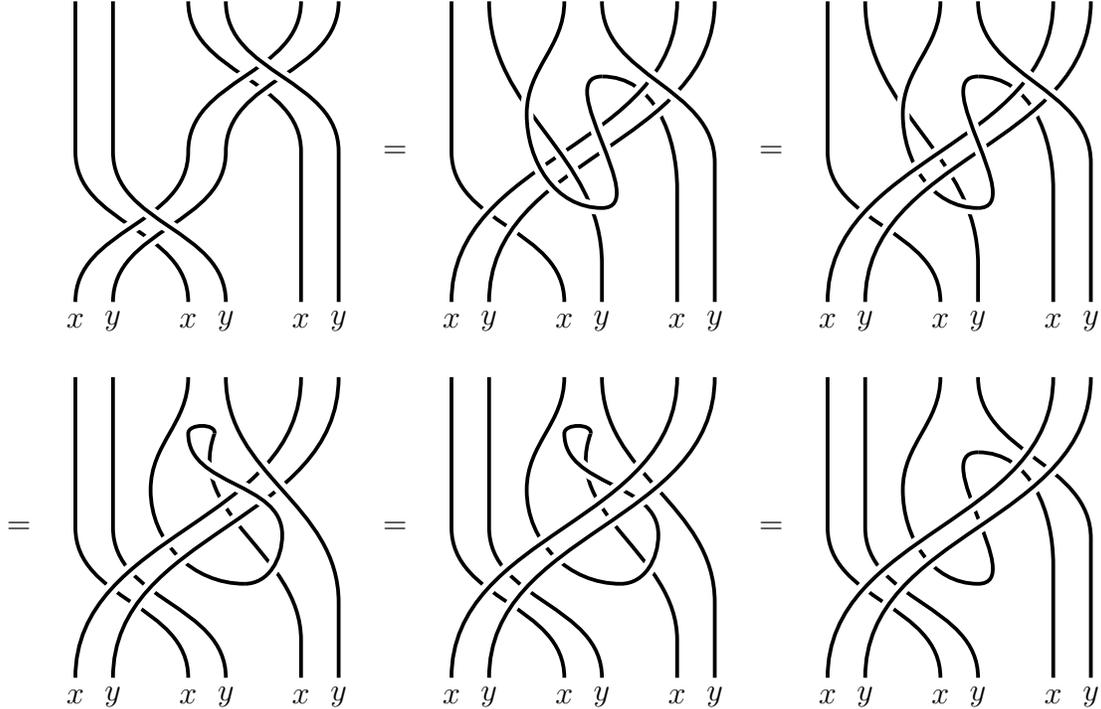
\[
 	\begin{tikzpicture}[scale=0.5]
 	\begin{scope}
 	\draw [line width=0.5mm] (4,0) to [out=90, in=-90] (1,4) to [out=90, in=-90] (1,8); 
 	\draw [line width=0.5mm] (7,0) to [out=90, in=-90] (7,4) to [out=90, in=-90] (4,8); 
 	\draw[white, line width=1.8mm] (1,0) to [out=90, in=-90] (4,4) to [out=90, in=-90] (7,8);
 	\draw[line width=0.5mm] (1,0) to [out=90, in=-90] (4,4) to [out=90, in=-90] (7,8); 
 	\draw[white, line width=1.8mm] (2,0) to [out=90, in=-90] (5,4) to [out=90, in=-90] (8,8);
 	\draw[line width=0.5mm] (2,0) to [out=90, in=-90] (5,4) to [out=90, in=-90] (8,8);
 	\draw[white, line width=1.8mm] (5,0) to [out=90, in=-90] (2,4) to [out=90, in=-90] (2,8);
 	\draw[line width=0.5mm] (5,0) to [out=90, in=-90] (2,4) to [out=90, in=-90] (2,8);
 	\draw[white, line width=1.8mm] (8,0) to [out=90, in=-90] (8,4) to [out=90, in=-90] (5,8);
 	\draw[line width=0.5mm] (8,0) to [out=90, in=-90] (8,4) to [out=90, in=-90] (5,8);
 	\node (x) at (1,-0.5) {$x$};
 	\node (x) at (2,-0.5) {$y$};
 	\node (x) at (4,-0.5) {$x$};
 	\node (x) at (5,-0.5) {$y$};
 	\node (x) at (7,-0.5) {$x$};
 	\node (x) at (8,-0.5) {$y$};
 	\end{scope}
 	\begin{scope}[xshift=9.5cm]
 	\node (k) at (0,4) {$=$};
 	\end{scope}
 	\begin{scope}[xshift=10cm]
 	\draw [line width=0.5mm] (4,0) to [out=90, in=-90] (1,4) to [out=90, in=-90] (1,8); 
 	\draw [line width=0.5mm] (7,0) to [out=90, in=-90] (7,3) to [out=90, in=0] (5,6); 
 	\draw[white, line width=1.8mm] (1,0) to [out=90, in=-90] (7,8);
 	\draw[line width=0.5mm] (1,0) to [out=90, in=-90] (7,8); 
 	\draw[white, line width=1.8mm] (2,0) to [out=90, in=-90] (8,8);
 	\draw[line width=0.5mm] (2,0) to [out=90, in=-90] (8,8);
 	\draw[white, line width=1.8mm] (5,0) to [out=90, in=-90] (5,1) to [out=90, in=-90] (2,8);
 	\draw[line width=0.5mm] (5,0) to [out=90, in=-90] (5,1) to [out=90, in=-90] (2,8);
 	\draw [white, line width=1.8mm] (5,6) to [out=180, in=0] (5,2.5) to [out=180, in=-90] (3,5) to [out=90, in=-90] (4,8); 
 	\draw [line width=0.5mm] (5,6) to [out=180, in=0] (5,2.5) to [out=180, in=-90] (3,5) to [out=90, in=-90] (4,8);
 	\draw[white, line width=1.8mm] (8,0) to [out=90, in=-90] (8,3.7) to [out=90, in=-90] (5,8);
 	\draw[line width=0.5mm] (8,0) to [out=90, in=-90] (8,3.7) to [out=90, in=-90] (5,8);
 	\node (x) at (1,-0.5) {$x$};
 	\node (x) at (2,-0.5) {$y$};
 	\node (x) at (4,-0.5) {$x$};
 	\node (x) at (5,-0.5) {$y$};
 	\node (x) at (7,-0.5) {$x$};
 	\node (x) at (8,-0.5) {$y$};
 	\end{scope}
 	\begin{scope}[xshift=19.5cm]
 	\node (k) at (0,4) {$=$};
 	\end{scope}
 	\begin{scope}[xshift=20cm]
 	\draw [line width=0.5mm] (4,0) to [out=90, in=-90] (1,4) to [out=90, in=-90] (1,8); 
 	\draw[white, line width=1.8mm] (5,0) to [out=90, in=-90] (5,1) to [out=90, in=-90] (2,8);
 	\draw[line width=0.5mm] (5,0) to [out=90, in=-90] (5,1) to [out=90, in=-90] (2,8);
 	\draw [white, line width=1.8mm] (5,2.5) to [out=180, in=-90] (3,5) to [out=90, in=-90] (4,8); 
 	\draw [line width=0.5mm] (5,6) to [out=180, in=0] (5,2.5) to [out=180, in=-90] (3,5) to [out=90, in=-90] (4,8);
 	\draw [line width=0.5mm] (7,0) to [out=90, in=-90] (7,3) to [out=90, in=0] (5,6); 
 	\draw[white, line width=1.8mm] (1,0) to [out=90, in=-90] (7,8);
 	\draw[line width=0.5mm] (1,0) to [out=90, in=-90] (7,8); 
 	\draw[white, line width=1.8mm] (2,0) to [out=90, in=-90] (8,8);
 	\draw[line width=0.5mm] (2,0) to [out=90, in=-90] (8,8);
 	\draw [white, line width=1.8mm] (5,6) to [out=180, in=0] (5,2.5); 
 	\draw [line width=0.5mm] (5,6) to [out=180, in=0] (5,2.5);
 	\draw[white, line width=1.8mm] (8,0) to [out=90, in=-90] (8,3.7) to [out=90, in=-90] (5,8);
 	\draw[line width=0.5mm] (8,0) to [out=90, in=-90] (8,3.7) to [out=90, in=-90] (5,8);
 	\node (x) at (1,-0.5) {$x$};
 	\node (x) at (2,-0.5) {$y$};
 	\node (x) at (4,-0.5) {$x$};
 	\node (x) at (5,-0.5) {$y$};
 	\node (x) at (7,-0.5) {$x$};
 	\node (x) at (8,-0.5) {$y$};
 	\end{scope}
 	\begin{scope}[xshift=-0.5cm, yshift=-10cm]
 	\node (k) at (0,4) {$=$};
 	\end{scope}
 	\begin{scope}[xshift=0cm, yshift=-10cm]
 	\draw [line width=0.5mm] (4,0) to [out=90, in=-90] (1,4) to [out=90, in=-90] (1,8); 
 	\draw[line width=0.5mm] (5,0) to [out=90, in=-90] (2,4) to [out=90, in=-90] (2,8);
 	\draw [white, line width=1.8mm] (5.5,2.5) to [out=180, in=-90] (3,5) to [out=90, in=-90] (4,8); 
 	\draw [line width=0.5mm] (5.5,2.5) to [out=180, in=-90] (3,5) to [out=90, in=-90] (4,8);
 	\draw [line width=0.5mm] (7,0) to [out=90, in=-90] (7,1) to [out=90, in=-110] (4.7,6.5); 
 	\draw[white, line width=1.8mm] (1,0) to [out=90, in=-90] (7,8);
 	\draw[line width=0.5mm] (1,0) to [out=90, in=-90] (7,8); 
 	\draw[white, line width=1.8mm] (2,0) to [out=90, in=-90] (8,8);
 	\draw[line width=0.5mm] (2,0) to [out=90, in=-90] (8,8);
 	\draw [white, line width=1.8mm] (4.7,6.5) to [out=90, in=90] (4,6.5) to [out=-90, in=90] (6.5,3.8) to [out=-90, in=0] (5.5,2.5); 
 	\draw [line width=0.5mm] (4.7,6.5) to [out=80, in=90] (4,6.5) to [out=-90, in=90] (6.5,3.8) to [out=-90, in=0] (5.5,2.5);
 	\draw[white, line width=1.8mm] (8,0) to [out=90, in=-90] (8,2) to [out=90, in=-90] (5,8);
 	\draw[line width=0.5mm] (8,0) to [out=90, in=-90] (8,2) to [out=90, in=-90] (5,8);
 	\node (x) at (1,-0.5) {$x$};
 	\node (x) at (2,-0.5) {$y$};
 	\node (x) at (4,-0.5) {$x$};
 	\node (x) at (5,-0.5) {$y$};
 	\node (x) at (7,-0.5) {$x$};
 	\node (x) at (8,-0.5) {$y$};
 	\end{scope}
 	\begin{scope}[xshift=9.5cm, yshift=-10cm]
 	\node (k) at (0,4) {$=$};
 	\end{scope}
 	\begin{scope}[xshift=10cm, yshift=-10cm]
 	\draw [line width=0.5mm] (4,0) to [out=90, in=-90] (1,4) to [out=90, in=-90] (1,8); 
 	\draw[line width=0.5mm] (5,0) to [out=90, in=-90] (2,4) to [out=90, in=-90] (2,8);
 	\draw [white, line width=1.8mm] (5.5,2.5) to [out=180, in=-90] (3,5) to [out=90, in=-90] (4,8); 
 	\draw [line width=0.5mm] (5.5,2.5) to [out=180, in=-90] (3,5) to [out=90, in=-90] (4,8);
 	\draw [line width=0.5mm] (7,0) to [out=90, in=-90] (7,1) to [out=90, in=-110] (4.7,6.5); 
 	\draw [white, line width=1.8mm] (4.7,6.5) to [out=90, in=90] (4,6.5) to [out=-90, in=90] (6.5,3.8) to [out=-90, in=0] (5.5,2.5); 
 	\draw [line width=0.5mm] (4.7,6.5) to [out=80, in=90] (4,6.5) to [out=-90, in=90] (6.5,3.8) to [out=-90, in=0] (5.5,2.5);
 	\draw[white, line width=1.8mm] (8,0) to [out=90, in=-90] (8,2) to [out=90, in=-90] (5,8);
 	\draw[line width=0.5mm] (8,0) to [out=90, in=-90] (8,2) to [out=90, in=-90] (5,8);
 	\draw[white, line width=1.8mm] (1,0) to [out=90, in=-90] (7,8);
 	\draw[line width=0.5mm] (1,0) to [out=90, in=-90] (7,8); 
 	\draw[white, line width=1.8mm] (2,0) to [out=90, in=-90] (8,8);
 	\draw[line width=0.5mm] (2,0) to [out=90, in=-90] (8,8);
 	\node (x) at (1,-0.5) {$x$};
 	\node (x) at (2,-0.5) {$y$};
 	\node (x) at (4,-0.5) {$x$};
 	\node (x) at (5,-0.5) {$y$};
 	\node (x) at (7,-0.5) {$x$};
 	\node (x) at (8,-0.5) {$y$};
 	\end{scope}
 	\begin{scope}[xshift=19.5cm, yshift=-10cm]
 	\node (k) at (0,4) {$=$};
 	\end{scope}
 	\begin{scope}[xshift=20cm, yshift=-10cm]
 	\draw [line width=0.5mm] (4,0) to [out=90, in=-90] (1,4) to [out=90, in=-90] (1,8); 
 	\draw[line width=0.5mm] (5,0) to [out=90, in=-90] (2,4) to [out=90, in=-90] (2,8);
 	\draw [line width=0.5mm] (7,0) to [out=90, in=-90] (7,3) to [out=90, in=0] (5,6); 
 	\draw [white, line width=1.8mm] (5,6) to [out=180, in=0] (5,2.5) to [out=180, in=-90] (3,5) to [out=90, in=-90] (4,8); 
 	\draw [line width=0.5mm] (5,6) to [out=180, in=0] (5,2.5) to [out=180, in=-90] (3,5) to [out=90, in=-90] (4,8);
 	\draw[white, line width=1.8mm] (8,0) to [out=90, in=-90] (8,3.7) to [out=90, in=-90] (5,8);
 	\draw[line width=0.5mm] (8,0) to [out=90, in=-90] (8,3.7) to [out=90, in=-90] (5,8);
 	\draw[white, line width=1.8mm] (1,0) to [out=90, in=-90] (7,8);
 	\draw[line width=0.5mm] (1,0) to [out=90, in=-90] (7,8); 
 	\draw[white, line width=1.8mm] (2,0) to [out=90, in=-90] (8,8);
 	\draw[line width=0.5mm] (2,0) to [out=90, in=-90] (8,8);
 	\node (x) at (1,-0.5) {$x$};
 	\node (x) at (2,-0.5) {$y$};
 	\node (x) at (4,-0.5) {$x$};
 	\node (x) at (5,-0.5) {$y$};
 	\node (x) at (7,-0.5) {$x$};
 	\node (x) at (8,-0.5) {$y$};
 	\end{scope}
 	\end{tikzpicture}
 	\]
 	The last diagram is actually equal to $\widetilde{s}_i\widetilde{s}_{i+1}$. In the second and fourth equalities, we push the double lines through the crossings using the symetric relation. 
 \end{proof}

 In order to compute  $\text{End}((x\otimes y)^{\otimes n})$ explicitly, we assume $\mathcal{C}$ is a \textit{ribbon fusion category}  (\cite{EGNO}). In this setting, the objects decompose into direct sum of simple objects.  Then for any simple object $z$, $\text{Hom}(z, (x\otimes y)^{\otimes n})$ is a linear representation of $\text{LB}_n$. By fusion rules, $x\otimes y=\oplus z_i$ for some simple objects $z_i$'s. Suppose each $z_i$ is bosonic or fermionic, then $x\otimes y$ has trivial double braiding. Recall that an object $z$ is bosonic/fermionic if $z\otimes z^*=\mathbf{1}$ and its twist or topological spin is $+1/-1$.  The followings are examples along this line.
 
 ~\\
\textbf{ Example 1.} Ising category.
 Let us consider the Ising category whose simple objects are $\mathbf{1}$, $\sigma$ and $\psi$ with twists $1$, $e^{2\pi i/16}$, and $-1$, respectively.  Its fusion rules are $\sigma\otimes\sigma=\mathbf{1}\oplus\psi$, $\psi\otimes\psi=\mathbf{1}$ and $\psi\otimes\sigma=\sigma\otimes\psi=\sigma$. Therefore,  $\text{Hom}(1, (\sigma\otimes\sigma)^{\otimes n})$ and $\text{Hom}(\psi, (\sigma\otimes\sigma)^{\otimes n})$ are both $\text{LB}_n$ representations of dimension $2^{n-1}$.
 
 ~\\
\textbf{Example 2.} Tambara-Yamagami categories.  \cite{TY} Given a finite abelian group $G=\{g_i\}$, the simple objects of the TY category $\mathcal{TY}(G)$ are $g_i$'s together with $m$. Its fusion rules are group multiplication and $mg=gm=m$, $m\otimes m=\oplus g_i$. It turns out in \cite{Siehler} that $\mathcal{TY}(G)$ is braided if and only if $G$ is an elementary abelian 2-group. That is, $g_i^2=1$ for any element $g_i$ in $G$ and so $g_i$'s twists are all $\pm 1$. Therefore, by the last fusion rule, $\text{Hom}(m,(m\otimes m)^{\otimes n})$ gives a $LB_n$ representation.

  ~\\
\hspace*{0.3cm} In the rest of this section, we apply graph calculus to derive formulas for computing the general representing matrices. 
 For simple objects $x$, $y$ and $z$, a fusion basis for $\text{Hom}(z, (x\otimes y)^{\otimes n})$ is chosen as
 
 \[
 \begin{tikzpicture}[scale=0.4]
 \begin{scope}
 \draw [line width=0.5mm] (0,0) to (5,-5); 
 \draw [dashed, line width=0.5mm] (5,-5) to (6.5,-6.5); 
 \draw [line width=0.5mm] (6.5,-6.5) to (7,-7); 
 \draw [line width=0.5mm] (0,0) to (-1,1); 
 \draw [line width=0.5mm] (0,0) to (1,1); 
 \draw [line width=0.5mm] (2,-2) to (5,1); 
 \draw [line width=0.5mm] (4,0) to (3,1); 
 \draw [line width=0.5mm] (4,-4) to (9,1); 
 \draw [line width=0.5mm] (8,0) to (7,1); 
 \node (k) at (-1,1.6) {\Large $x$};
 \node (k) at (1,1.6) {\Large $y$};
 \node (k) at (3,1.6) {\Large $x$};
 \node (k) at (5,1.6) {\Large $y$};
 \node (k) at (7,1.6) {\Large $x$};
 \node (k) at (9,1.6) {\Large $y$};
 \node (k) at (0.3,-1.3) {\Large $a_1$};
 \node (k) at (3.9,-1.3) {\Large $a_2$};
 \node (k) at (7.9,-1.3) {\Large $a_3$};
 \node (k) at (2.3,-3.3) {\Large $b_2$};
 \node (k) at (4.3,-5.3) {\Large $b_3$};
 \node (k) at (7,-7.5) {\Large $z$};
 \end{scope}
 \end{tikzpicture}
 \]	
 
Here we only consider multiplicity-free fusion rules for simplicity. The general case is completely similar. The $\text{LB}_n$ will be calculated in terms of  the $F$-symbols and $R$-symbols as follows. They are the matrix entries representing the categorical associtivity and commutativity.
  
  \[
  \begin{tikzpicture}[scale=0.5]
  \begin{scope}[xshift=-6cm, yshift=2cm, scale=0.6]
  \begin{scope}[xshift=-12cm]
  \draw [line width=0.5mm] (1,-1) to (5,-5); 
  \draw [ line width=0.5mm] (5,-5) to (6.5,-6.5); 
  \draw [ line width=0.5mm] (1,-1) to (0,0); 
  \draw [line width=0.5mm] (2,-2) to (4,0); 
  \draw [line width=0.5mm] (4,-4) to (8,0); 
  \node (k) at (4,0.6) {\Large $b$};
  \node (k) at (8,0.6) {\Large $c$};
  \node (k) at (0,0.6) {\Large $a$};
  \node (k) at (2.4,-3.5) {\Large $m$};
  \node (k) at (6.5,-7) {\Large $z$};
  \end{scope}
  \node (k) at (0,-3) {$=\sum\limits_{n}F^{abc}_{z;nm}$};
  \begin{scope}[xshift=4cm]
  \draw [line width=0.5mm] (1,-1) to (5,-5); 
  \draw [ line width=0.5mm] (5,-5) to (6.5,-6.5); 
  \draw [ line width=0.5mm] (1,-1) to (0,0); 
  \draw [line width=0.5mm] (6,-2) to (4,0); 
  \draw [line width=0.5mm] (4,-4) to (8,0); 
  \node (k) at (4,0.6) {\Large $b$};
  \node (k) at (8,0.6) {\Large $c$};
  \node (k) at (0,0.6) {\Large $a$};
  \node (k) at (5.6,-3.5) {\Large $n$};
  \node (k) at (6.5,-7) {\Large $z$};
  \end{scope}
  \end{scope}
  \begin{scope}[xshift=3cm, scale=0.6]
  \node (k) at (0,0) {$\large ,$};
  \end{scope}
  \begin{scope}[xshift=8cm]
  \begin{scope}[xshift=-2.5cm]
  \draw[line width=0.5mm] (0,0)--(0,-2); 
  \draw [line width=0.5mm] (0,0) to [out=45, in=-45] (-1,2); 
  \draw [white, line width=2mm] (0,0) to [out=160, in=-160] (1,2); 
  \draw [line width=0.5mm] (0,0) to [out=160, in=-160] (1,2); 
  \node (x) at (0,-2.5) {\Large $z$};
  \node (x) at (-1,2.5) {\Large $a$};
  \node (x) at (1,2.5) {\Large $b$};
  \end{scope}
  \begin{scope}
  \node (k) at (0,0) {$=R^{ab}_z$};
  \end{scope}
  \begin{scope}[xshift=2.5cm]
  \draw[line width=0.5mm] (0,0)--(0,-2); 
  \draw [line width=0.5mm] (0,0) to (1,2); 
  \draw [line width=0.5mm] (0,0) to (-1,2); 
  \node (x) at (0,-2.5) {\Large $z$};
  \node (x) at (-1,2.5) {\Large $a$};
  \node (x) at (1,2.5) {\Large $b$};
  \end{scope}
  \end{scope}	 
  \end{tikzpicture}
  \]

 As a result,  The action of $\widetilde{\sigma}_{i}$ and $\widetilde{s}_j$ on the basis vectors are given by

\[
\begin{tikzpicture}[scale=0.5]
\begin{scope}[xshift=-2cm]
\node (k) at (0,-1) {\Large $\widetilde{s}_j$};
\node (<) at (0.8,-1) {$\Bigg($};
\node (>) at (10.8,-1) {$\Bigg)$};
\end{scope}
\begin{scope}
\draw [line width=0.5mm] (1,-1) to (5,-5); 
\draw [dashed, line width=0.5mm] (5,-5) to (6.5,-6.5); 
\draw [dashed, line width=0.5mm] (1,-1) to (0,0); 
\draw [line width=0.5mm] (2,-2) to (3,-1); 
\draw [line width=0.5mm] (3,-1) to [out=45, in=-90] (4,0); 
\draw [line width=0.5mm] (3,-1) to [out=135, in=-90] (2,0); 
\draw [line width=0.5mm] (4,-4) to (7,-1); 
\draw [line width=0.5mm] (7,-1) to [out=45, in=-90] (8,0); 
\draw [line width=0.5mm] (7,-1) to [out=135, in=-90] (6,0); 
\draw[line width=0.5mm] (6,0) to [out=90, in=-90] (2,4); 
\draw[line width=0.5mm] (8,0) to [out=90, in=-90] (4,4); 
\draw[white, line width=2mm] (2,0) to [out=90, in=-90] (6,4); 
\draw[line width=0.5mm] (2,0) to [out=90, in=-90] (6,4); 
\draw[white, line width=2mm] (4,0) to [out=90, in=-90] (8,4); 
\draw[line width=0.5mm] (4,0) to [out=90, in=-90] (8,4); 
\node (k) at (2,4.6) {\Large $x$};
\node (k) at (4,4.6) {\Large $y$};
\node (k) at (6,4.6) {\Large $x$};
\node (k) at (8,4.6) {\Large $y$};
\node (k) at (0.4,-1.5) {\Large $b_{j-1}$};
\node (k) at (3,-2) {\Large $a_{j}$};
\node (k) at (7,-2.6) {\Large $a_{j+1}$};
\node (k) at (2.4,-3.5) {\Large $b_{j}$};
\node (k) at (4.2,-5.3) {\Large $b_{j+1}$};
\end{scope}
\begin{scope}[xshift=9.5cm]
\node (k) at (0.4,-1) {$=$};
\end{scope}
\begin{scope}[xshift=11.5cm]
\draw [line width=0.5mm] (1,-1) to (5,-5); 
\draw [dashed, line width=0.5mm] (5,-5) to (6.5,-6.5); 
\draw [dashed, line width=0.5mm] (1,-1) to (0,0); 
\draw [line width=0.5mm] (2,-2) to [out=45, in=-90] (3,-1); 
\draw [line width=0.5mm] (3,2.5) to [out=45, in=-90] (4,4); 
\draw [line width=0.5mm] (3,2.5) to [out=135, in=-90] (2,4); 
\draw [line width=0.5mm] (4,-4) to  [out=45, in=-90] (7,-1); 
\draw [line width=0.5mm] (7,2.5) to [out=45, in=-90] (8,4); 
\draw [line width=0.5mm] (7,2.5) to [out=135, in=-90] (6,4); 
\draw[line width=0.5mm] (7,-1) to [out=90, in=-90] (3,2.5); 
\draw[white, line width=2mm] (3,-1) to [out=90, in=-90] (7,2.5); 
\draw[line width=0.5mm] (3,-1) to [out=90, in=-90] (7,2.5); 
\node (k) at (2,4.6) {\Large $x$};
\node (k) at (4,4.6) {\Large $y$};
\node (k) at (6,4.6) {\Large $x$};
\node (k) at (8,4.6) {\Large $y$};
\node (k) at (0.4,-1.5) {\Large $b_{j-1}$};
\node (k) at (3.7,-1.5) {\Large $a_{j}$};
\node (k) at (8.2,-1.5) {\Large $a_{j+1}$};
\node (k) at (2.4,-3.5) {\Large $b_{j}$};
\node (k) at (4.2,-5.3) {\Large $b_{j+1}$};
\end{scope}
\begin{scope}[xshift=5.5cm, yshift=-12cm]
\node (k) at (0,-1) {\Large $=\sum\limits_{c_j,b'_j}F^{b_{j-1}a_{j}a_{j+1}}_{b_{j+1};c_jb_j}R^{a_{j+1}a_{j}}_{c_{j}}\bar{F}^{b_{j-1}a_{j+1}a_{j}}_{b_{j+1};b'_jc_j}$};
\node (k) at (18,-1) {\Large ,};
\end{scope}
\begin{scope}[xshift=14.5cm, yshift=-9cm]
\draw [line width=0.5mm] (1,-1) to (5,-5); 
\draw [dashed, line width=0.5mm] (5,-5) to (6.5,-6.5); 
\draw [dashed, line width=0.5mm] (1,-1) to (0,0); 
\draw [line width=0.5mm] (2,-2) to (3,-1); 
\draw [line width=0.5mm] (3,-1) to [out=45, in=-90] (4,0); 
\draw [line width=0.5mm] (3,-1) to [out=135, in=-90] (2,0); 
\draw [line width=0.5mm] (4,-4) to (7,-1); 
\draw [line width=0.5mm] (7,-1) to [out=45, in=-90] (8,0); 
\draw [line width=0.5mm] (7,-1) to [out=135, in=-90] (6,0); 
\node (k) at (2,0.6) {\Large $x$};
\node (k) at (4,0.6) {\Large $y$};
\node (k) at (6,0.6) {\Large $x$};
\node (k) at (8,0.6) {\Large $y$};
\node (k) at (0.4,-1.5) {\Large $b_{j-1}$};
\node (k) at (3.7,-1.8) {\Large $a_{j+1}$};
\node (k) at (6.3,-2.8) {\Large $a_{j}$};
\node (k) at (2.4,-3.5) {\Large $b'_{j}$};
\node (k) at (4.2,-5.3) {\Large $b_{j+1}$};
\end{scope}
\end{tikzpicture}
\]	

\[
\begin{tikzpicture}[scale=0.5]
\begin{scope}[xshift=-2cm]
\node (k) at (0,-1) {\Large $\widetilde{\sigma}_i$};
\node (<) at (0.8,-1) {$\Bigg($};
\node (>) at (10.8,-1) {$\Bigg)$};
\end{scope}
\begin{scope}
\draw [line width=0.5mm] (1,-1) to (5,-5); 
\draw [dashed, line width=0.5mm] (5,-5) to (6.5,-6.5); 
\draw [dashed, line width=0.5mm] (1,-1) to (0,0); 
\draw [line width=0.5mm] (2,-2) to (3,-1); 
\draw [line width=0.5mm] (3,-1) to [out=45, in=-90] (4,0); 
\draw [line width=0.5mm] (3,-1) to [out=135, in=-90] (2,0); 
\draw [line width=0.5mm] (4,-4) to (7,-1); 
\draw [line width=0.5mm] (7,-1) to [out=45, in=-90] (8,0); 
\draw [line width=0.5mm] (7,-1) to [out=135, in=-90] (6,0); 
\draw[line width=0.5mm] (6,0) to [out=90, in=-90] (2,4); 
\draw[white, line width=2mm] (2,0) to [out=90, in=-90] (6,4); 
\draw[line width=0.5mm] (2,0) to [out=90, in=-90] (6,4); 
\draw[white, line width=2mm] (4,0) to [out=90, in=-90] (8,4); 
\draw[line width=0.5mm] (4,0) to [out=90, in=-90] (8,4); 
\draw[white, line width=2mm] (8,0) to [out=90, in=-90] (4,4); 
\draw[line width=0.5mm] (8,0) to [out=90, in=-90] (4,4); 
\node (k) at (2,4.6) {\Large $x$};
\node (k) at (4,4.6) {\Large $y$};
\node (k) at (6,4.6) {\Large $x$};
\node (k) at (8,4.6) {\Large $y$};
\node (k) at (0.4,-1.5) {\Large $b_{i-1}$};
\node (k) at (3,-2) {\Large $a_{i}$};
\node (k) at (7,-2.6) {\Large $a_{i+1}$};
\node (k) at (2.4,-3.5) {\Large $b_{i}$};
\node (k) at (4.2,-5.3) {\Large $b_{i+1}$};
\end{scope}
\begin{scope}[xshift=9.5cm]
\node (k) at (0.4,-1) {$=$};
\end{scope}
\begin{scope}[xshift=11.5cm]
\draw [line width=0.5mm] (1,-1) to (5,-5); 
\draw [dashed, line width=0.5mm] (5,-5) to (6.5,-6.5); 
\draw [dashed, line width=0.5mm] (1,-1) to (0,0); 
\draw [line width=0.5mm] (4,-4) to [out=60, in=-90] (6,-2); 
\draw [line width=0.5mm] (6,-2) to [out=120, in=-90] (5,-1) to [out=90, in=-90] (2,4); 
\draw [white, line width=2mm] (2,-2) to  [out=60, in=-90] (7,2.5); 
\draw [line width=0.5mm] (2,-2) to  [out=60, in=-90] (7,2.5); 
\draw [white, line width=2mm] (6,-2) to [out=60, in=-90] (7,-1) to [out=90, in=-90] (4,4); 
\draw [line width=0.5mm] (6,-2) to [out=60, in=-90] (7,-1) to [out=90, in=-90] (4,4); 
\draw [line width=0.5mm] (7,2.5) to [out=45, in=-90] (8,4); 
\draw [line width=0.5mm] (7,2.5) to [out=135, in=-90] (6,4); 
\node (k) at (2,4.6) {\Large $x$};
\node (k) at (4,4.6) {\Large $y$};
\node (k) at (6,4.6) {\Large $x$};
\node (k) at (8,4.6) {\Large $y$};
\node (k) at (0.4,-1.5) {\Large $b_{i-1}$};
\node (k) at (3.5,-1.5) {\Large $a_{i}$};
\node (k) at (6.5,-3.5) {\Large $a_{i+1}$};
\node (k) at (2.4,-3.5) {\Large $b_{i}$};
\node (k) at (4.2,-5.3) {\Large $b_{i+1}$};
\end{scope}
\begin{scope}[xshift=5.5cm, yshift=-12cm]
\node (k) at (0,-1) {\Large $=\sum\limits_{c,k,b,p,b'_{i},a'_{i}}\eta(c,k,b,p,b'_{i},a'_{i})$};
\node (k) at (18,-1) {\Large .};
\end{scope}
\begin{scope}[xshift=13cm, yshift=-9cm]
\draw [line width=0.5mm] (1,-1) to (5,-5); 
\draw [dashed, line width=0.5mm] (5,-5) to (6.5,-6.5); 
\draw [dashed, line width=0.5mm] (1,-1) to (0,0); 
\draw [line width=0.5mm] (2,-2) to (3,-1); 
\draw [line width=0.5mm] (3,-1) to [out=45, in=-90] (4,0); 
\draw [line width=0.5mm] (3,-1) to [out=135, in=-90] (2,0); 
\draw [line width=0.5mm] (4,-4) to (7,-1); 
\draw [line width=0.5mm] (7,-1) to [out=45, in=-90] (8,0); 
\draw [line width=0.5mm] (7,-1) to [out=135, in=-90] (6,0); 
\node (k) at (2,0.6) {\Large $x$};
\node (k) at (4,0.6) {\Large $y$};
\node (k) at (6,0.6) {\Large $x$};
\node (k) at (8,0.6) {\Large $y$};
\node (k) at (0.4,-1.5) {\Large $b_{i-1}$};
\node (k) at (3.2,-1.8) {\Large $a'_{i}$};
\node (k) at (6.3,-2.8) {\Large $a_{i}$};
\node (k) at (2.4,-3.5) {\Large $b'_{i}$};
\node (k) at (4.2,-5.3) {\Large $b_{i+1}$};
\end{scope}
\end{tikzpicture}
\]
where $\eta(c,k,b,p,b'_{i},a'_{i})=\bar{F}^{b_{i}xy}_{b_{i+1};ca_{i+1}}F^{b_{i-1}xa_{i}}_{c;kb_i}R^{xa_{i}}_{k}\bar{F}^{b_{i-1}xa_{i}}_{c;bk}F^{bya_{i}}_{b_{i+1};pc}\bar{R}^{ya_{i}}_{p}\bar{F}^{bya_{i}}_{b_{i+1};b'_{i}p}F^{b_{i-1}xy}_{b'_{i};a'_ib}$,  $\bar{F}$ is the inverse matrices of $F$ and $\bar{R}$ is the inverse of $R$.

\section{Conclusion}
In this paper, we propose the method of dimension reduction to construct loop braid group representations from braid tensor categories. The advantage of this method is that we can apply the well-developed theory of braid tensor categories rather than (3+1) TQFTs. The question becomes searching for two objects whose tensor product having trivial double braiding. This kind of objects are generally few since their quantum dimensions are square root of integers. Topologically, once having $\text{LB}_n$ represenations, we can continue to acquire the invariants for welded links in the spirit of \cite{Turaevknot}. This will be discussed in another paper.

Moreover, the dimension reduction approach can be used for more motion groups of links. Another typical example is the necklace braid groups. This name comes from the shape of the motion where the moving knots are linked to a base ring just like a necklace. The case that several circles are linked to another big circle was discussed in \cite{BKMR, WL}. Another case is that a sequence of torus knots (like trefoil knots) are linked to a base circle. This motion group was firstly studied in \cite{Gold} and should give more interesting and complicated braiding than that from the loop braid groups.

\vspace{1cm}
\noindent\textbf{ACKNOWLEDGEMENT} 

 The author acknowledges the support from NSFC Grant No. 11701293.
 
\vspace{1cm}
\noindent\textbf{DATA AVAILABILITY STATEMENTS} 

The data that supports the findings of this study are available within this article.

\end{document}